\documentclass[reqno, a4paper, 11pt, oneside]{article}
\usepackage[11pt]{extsizes}
\usepackage[a4paper,hmargin=2.5cm,vmargin=2.5cm, bindingoffset=0cm]{geometry}

\usepackage{mathtools}
\usepackage[utf8]{inputenc}
\usepackage[T2A]{fontenc}
\usepackage{amsmath,amssymb,amsthm}
\usepackage[english]{babel}
\usepackage{textcomp}
\usepackage[usenames]{color}
\usepackage{colortbl}
\usepackage{float}
\usepackage{relsize}
\usepackage{amsthm}

\usepackage{thmtools}
\usepackage{thm-restate}
\makeatletter
\newcommand*{\rom}[1]{\expandafter\@slowromancap\romannumeral #1@}
\makeatother

\usepackage[pagebackref,bookmarksopen]{hyperref}
\hypersetup{
	colorlinks=true,
	linkcolor=blue,
	citecolor=magenta,
	urlcolor=cyan,
}
\usepackage{cleveref}

\usepackage{graphicx}
\DeclareGraphicsExtensions{.pdf,.png,.jpg}

\usepackage{amsthm}

\theoremstyle{definition}

\newtheorem*{theorem*}{Theorem}
\newtheorem{theorem}{Theorem}[section]

\newtheorem*{Remark}{Remark}
\newtheorem{claim}[theorem]{Claim}
\newtheorem{lemma}[theorem]{Lemma}

\newtheorem{ex}[theorem]{Example}

\newtheorem{definition}[theorem]{Definition}

\newtheorem*{theorem-non}{Theorem}

\definecolor{applegreen}{rgb}{0.0, 0.42, 0.24}
\definecolor{applegreen}{rgb}{0.55, 0.71, 0.0}

\newcommand{\Z}{\mathbb{Z}}
\newcommand{\R}{\mathbb{R}}

\newcommand{\D}{\Delta}

\renewcommand{\t}{\tau}

\usepackage{abstract}
\addto{\captionsenglish}{}

\begin{document}

\begin{center}
		{\Large \sc
                B-facets in dimension 4
                }
	\vspace{3ex}
 
		{\sc{Fedor Selyanin}}

\end{center}

\begin{abstract}
We complete the classification of B-facets of a 4-dimensional Newton polyhedron, filling a gap in the classification of \cite{ELT22}, found by the authors of \cite{LPS22}.
\end{abstract}

\section{Introduction}

The local monodromy conjecture relates an arithmetic invariant of a singularity of a polynomial $f$ (namely, the Igusa, motivic or topological zeta-function) to a geometric invariant (the characteristic function $P_f$ of the Milnor monodromy, or Bernstein-Sato polynomial). The first version was posed in the \cite{I88}, but only recently there has been significant progress for singularities that are non-degenerate with respect to their Newton polyhedra: see \cite{ELT22} for arbitrary 4-dimensional polyhedra and \cite{LPS22} for simplicial polyhedra of arbitrary dimension.

A primitive degree $b$ root of unity has a chance to be a root or a pole of $P_f$, if some facet $\tau$ of the Newton polyhedron is cut by a support plane of the form $\sum_i a_ix_i=b,\, gcd(a)=1$. Checking whether it is really a root or a pole (i.e. whether the facet $\tau$ actually {\it contributes} roots to $P_f$) is an important part of the question.

One idea in \cite{ELT22} (expanding on the previous work in dimensions 2 and 3 \cite{L88}, \cite{LV11}) is to show that (i) a facet $\tau$ contributes a root to $P_f$ if it is not a $B$-facet (Definition \ref{B_facet_def} below), and (ii) a pair of adjacent facets $\phi$ and $\tau$ contributes a root to $P_f$ if they form a so-called border (Definition 5.1 in \cite{ELT22}). Then the proof of the conjecture in 4 dimensions requires to classify $B$-facets of 4-dimensional polyhedra. This was done in Lemma 5.18 of \cite{ELT22} claiming that all such facets are of two types (so-called $B_1$- and $B_2$-facets, see Theorem \ref{main_clas_th} below).

However, later the authors of \cite{LPS22} found examples contradicting the claimed classification \cite{LPS23}. Then the first author of \cite{ELT22} conjectured that every missing facet $\delta$ in this classification can be split into two pieces $\phi$ and $\tau$ satisfying the above mentioned definition of the border (except that they are two pieces of the same facet rather than two different facets) and thus do not require any further changes in the proof of the main result of \cite{ELT22} (see an upcoming erratum \cite{ELTe} for details).

The aim of this note is to prove this conjectural corrected classification.

\section{Classification theorem of B-facets in dimension $4$}

In our notations a polytope is a finite set, its dimension is the dimension of its convex hull, its face is the intersection of a face of its convex hull with the set, its facet is a face whose affine span has codimension $1$.

\begin{definition}\label{B_simplex_def}
    A $\mathbf{B}$\textbf{-simplex} in $\Z^n_{\ge 0}$ is a pyramid of height $1$ with base on a coordinate hyperplane. A zero-dimensional $B$-simplex is a point one of whose coordinates equals $1$. In this paper we use $n-1$ and $n$-dimensional B-simplices.
\end{definition}

\begin{definition}\label{B_facet_def}
    A bounded facet $\tau \subset \Z^n_{\ge 0}$, $dim(\tau) = n-1$ with positive normal covector is called a $\mathbf{B}$\textbf{-facet} if every $(n-1)$-dimensional simplex with vertices in it is $B$-simplex.
\end{definition}

In the whole text we assume that $\t$ is a $B$-facet. The following theorem is the main result of this paper.

\begin{theorem}\label{main_clas_th}
    Every B-facet $\tau$ is one of the following types:
\begin{enumerate}
    \item \textbf{$\mathbf{B_1}$-facet} i.e. a pyramid of height 1 with base on a coordinate hyperplane.
    \item \textbf{$\mathbf{B_2}$-facet} i.e. its projection on some coordinate 2-plane is the triangle $(0,0), (1,0), (0,1)$.
    \item \textbf{Flat border} i.e. it contains a triangle $ABC$ of the form 
    $$(0,0,\star, \star), (0,0,\star, \star), (1,1,\star, \star)$$
    and all the other points are contained in $\{x_1 = 0\} \cup \{x_2 = 0\}$.\\
    \item \textbf{Standard cross-polytope} i.e. the set 
    $$(1,1,0,0), (1,0,1,0), (1,0,0,1), (0,1,1,0), (0,1,0,1), (0,0,1,1)$$
\end{enumerate}

\end{theorem}

\begin{Remark}
    Every configuration described in the Theorem is always a $B$-facet except for flat borders, see the example below.
\end{Remark}

\begin{ex}
    The set 
    $$(0,0,0,5) (0,0,1,4) (1,1,0,3) (1,0,2,2) (0,1,2,2)$$
    is a flat border (even flat $B^2$-border from \cite[Definition 5.1]{ELT22}) but not a $B$-facet. 
\end{ex}

The example below is a generalization of an example discovered by \cite{LPS23}.

\begin{ex}\label{exotic_B_ex}
    There are examples of $B$-facets which are flat $B$-borders (i.e. flat borders, but not flat $B^2$-borders, see \cite[Definition 5.1]{ELT22}). For instance, there are two  such examples:
    \begin{enumerate}
        \item \textbf{Flat $\mathbf{B}$-border pyramid}, i.e. a set of 5 points of the form:
    $$(0,0,0,\star),(0,0,a,\star),(a,0,0,\star),(0,1,1,\star), (1,1,0,\star),$$
    where the last four points are contained in an affine 2-space.
    \item \textbf{Flat $\mathbf{B}$-border circuit}, i.e. a set of 5 points of the form:
    $$(0,0,0,\star),(1,1,0,0),(0,1,0,1),(\star,0,1,0), (0,0,\star,0)$$
    \end{enumerate}
    
\end{ex}

\subsection{Sketch of the proof: main lemmas which imply the Theorem}

The following claim is obvious since $\t$ has positive normal covector.

\begin{claim}\label{On_ray1_claim}
    If a unit point on a coordinate ray is in $B$-facet $\tau$ then $\tau$ is a $B_1$-facet.
\end{claim}

Thus, we assume that $\tau$ contains no unit points on coordinate rays.

\begin{definition}
    Face $F$ is called a \textbf{V-face} if it is contained in a coordinate subspace of dimension $dim(F)+1$. It is called \textbf{internal} if it does not contain any proper V-faces.
\end{definition}

\begin{claim}\label{internal_claim}
    Every bounded facet of Newton polyhedron is a V-face. Every V-face contains an internal V-face. Internal V-face is never a $B$-facet (in the corresponding coordinate subspace).
\end{claim}

We obtain Theorem from the following Lemmas.

\begin{lemma}\label{triangle}
    If a $B$-facet $\t\subset \Z^4$ contains an internal V-triangle then $\tau$ is either $B_1$-facet or (degenerated) $B_2$-facet or standard cross-polytope.
\end{lemma}

\begin{lemma}\label{segment}
    If a $B$-facet $\t\subset \Z^4$ contains a V-segment which is not a $B$-segment then it is either $B_1$-facet or $B_2$-facet or flat border.
\end{lemma}

\begin{lemma}\label{point}
    If a $B$-facet $\t\subset \Z^4$ is a pyramid with apex on a coordinate ray (and some 2-dimensional base) than it is either $B_1$-facet or $B_2$-facet or flat border (more precisely, flat $B$-border pyramid from Example \ref{exotic_B_ex}).
\end{lemma}

Let us prove that these Lemmas provide the Theorem \ref{main_clas_th}. It is sufficient to verify that every facet satisfies conditions of one of the Lemmas. Due to Claim \ref{internal_claim} every $B$-facet contains an internal V-triangle or V-segment or V-point. If it contains an internal V-triangle or V-segment then it satisfies Lemma \ref{triangle} or \ref{segment}. 

Suppose $\tau$ contains neither internal V-triangles nor internal V-segments. Then it contains an internal V-point i.e. a point on a coordinate ray. If the set of the rest points is $2$-dimensional then it satisfies Lemma \ref{point}. Otherwise some of the rest points form a V-tetrahedron. The V-tetrahedron contains an internal V-face which is neither a triangle nor a segment, so it is a V-point, i.e. a point on other coordinate ray. The segment between two points on coordinate rays is a V-segments, but not $B$-segment (by the assumption below Claim \ref{On_ray1_claim}) and satisfies Lemma \ref{segment}.

The rest of the text is devoted to proving the lemmas. Proof of Lemma \ref{triangle} is straight forward, Lemma \ref{segment} uses notion of \textbf{$\mathbf{B}$-polytopes} (our definition is different from \cite[Definition 1.4]{ELT22}). from \S \ref{B-pol} and Lemma \ref{point} uses \textbf{marked $\mathbf{B}$-polytopes}. Note that $B$-polytopes are not a special case of marked $B$-polytopes though the definitions are very similar. These definitions provide some tools for reducing the dimension of the problem (in arbitrary dimension) (see Lemma \ref{proj_lem} and Lemma \ref{section_lem}).

\section{$\mathbf {B}$-polytopes} \label{B-pol}

Note that this definition of $B$-polytope differs much from \cite[Definition 1.4]{ELT22}.

\begin{definition}\label{B_pol_def}
    A $\mathbf{B}$\textbf{-polytope} in $\Z^n_{\ge 0}$ is an $n$-dimensional set of lattice points such that every $(n-1)$-dimensional simplex with vertices in it is a $B$-simplex or a simplex whose affine span contains the origin. 
\end{definition}

\begin{claim}
 The only $B$-polytope in $\Z^1_{\ge 0}$ is the set $\{0,1\}$.
\end{claim}

For a coordinate subspace $E\subset \R^n$ denote by $p_{E}$ the projection along $E$ to the ortogonal complement $E^\bot$. 

\begin{lemma} \label{proj_lem}
    Consider a $B$-facet $\tau \subset \Z^n$. Consider a coordinate subspace $E\subset \R^n$ and suppose that $E\cap\tau$ is a V-face but not a $B$-face. Then the projection $p_E(\t)$ is a $B$-polytope.
\end{lemma}

\begin{proof}
    Assume the converse. Consider a $(n-\dim(E)-1)$-simplex $Q_{E^\bot} \subset p_E(\tau)$ and a $(\dim(E)-1)$-simplex $Q_E$ in $E$ which are not $B$-simplices, such that the affine span of $Q_{E^\bot}$ does not contain the origin. Then any set of points $Q_E \cup Q_{E^\bot}^\prime \subset \tau$ such that $p_E(Q_{E^\bot}^\prime) \to Q_{E^\bot}$ is one to one, forms a simplex (since $\text{aff} (Q_{E^\bot})$ does not contain the origin), but not a $B$-simplex.
\end{proof}

\begin{lemma}[$B$-polytopes in dimension $2$] \label{d-2}
Here is the classification of two-dimensional $B$-polytopes (up to reordering of coordinates):

\begin{enumerate}
    \item \textbf{$\mathbf{B_1}$-polytope}. Point $(a,1)$ and at least two points on the ray $(\star,0)$.
        \begin{figure}[H]
        \centering
        \includegraphics[scale = 1]{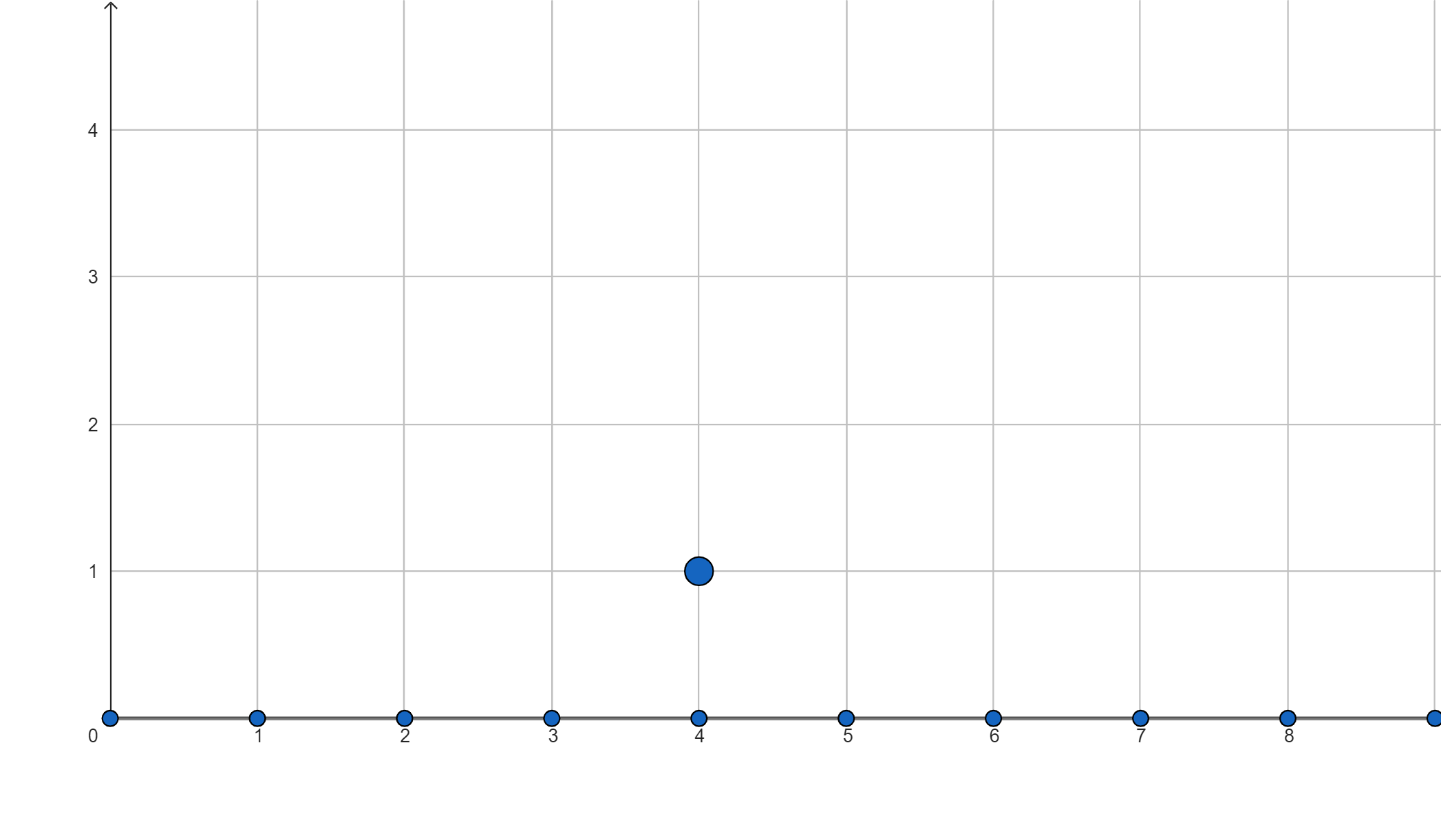}
        \caption{$B_1$-polytope}
        \label{B-p2-1}
    \end{figure}
    
    \item \textbf{Border polytope}. Points $(0,1),(1,1)$ and at least one point on the ray $(\ge 1,0)$ and maybe the origin.
        \begin{figure}[H]
        \centering
        \includegraphics[scale = 1]{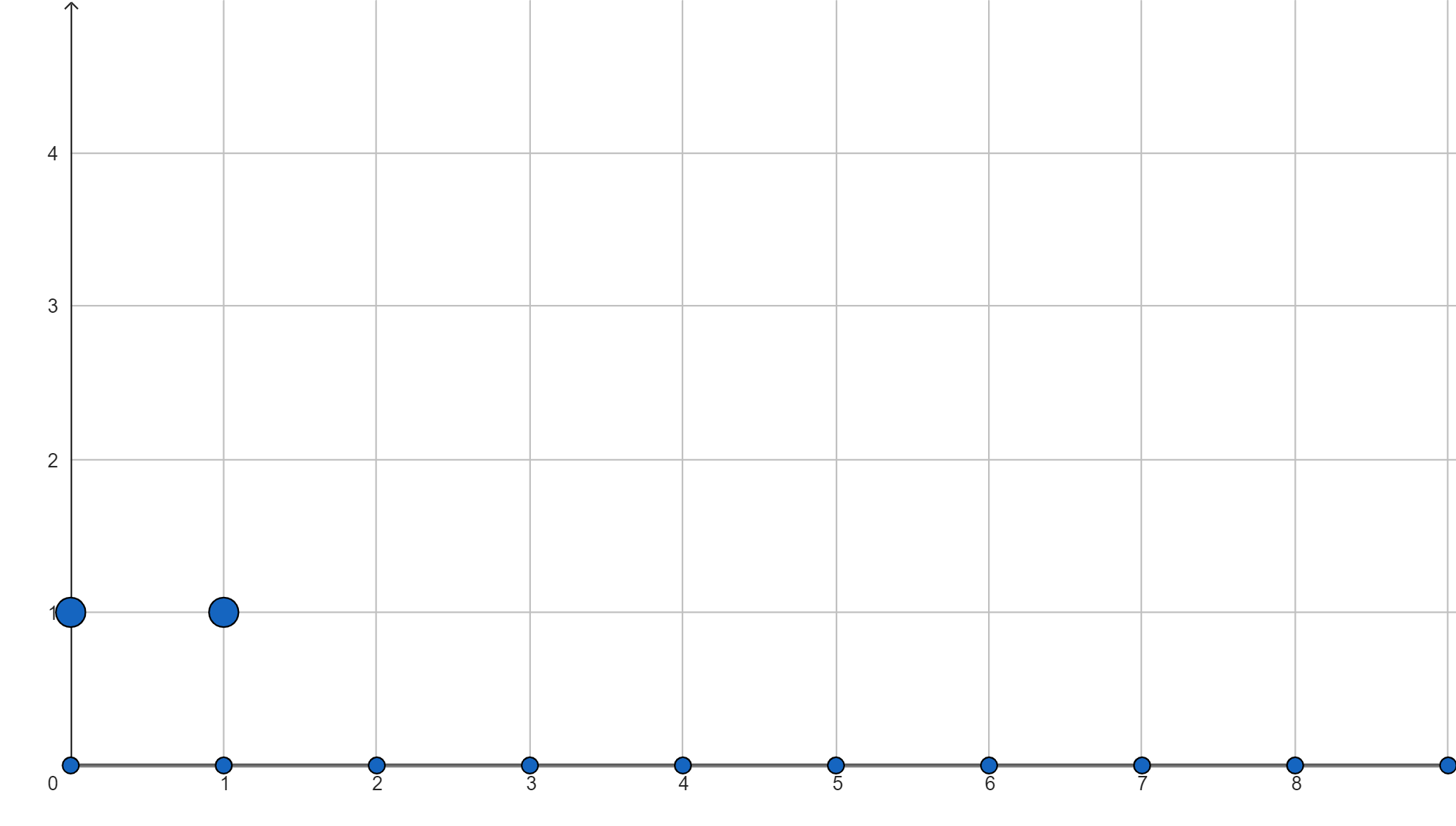}
        \caption{border polytope}
        \label{B-p2-2}
    \end{figure}
    
\end{enumerate}

\end{lemma}

\begin{proof}

If there is a coordinate ray such that there is only one point outside it then we obtain the case $1$. Otherwise we obtain the second case and the proof consists of the following observations.

\begin{enumerate}

    \item There is no point  $G \ (\ge 2, \ge 2)$.

    If there is such point $G \ (\ge 2,\ge 2)$ then there is another point $G^\prime$ such that affine span of $GG^\prime$ does not contain the origin since $B$-polytope is 2-dimensional. Then $GG^\prime$ is not $B$-segment and its affine span does not contain the origin.
    
    \item There is no point $G \ (\ge 2, 1)$ (and vice versa).

    If there is such point $G\ (\ge 2, 1)$, then consider another point $G^\prime$ of the form $(\star,\ge 1)$, it exists since there are two points outside $\{x_2 = 0\}$. The segment $GG^\prime$ is not $B$-segment and its affine span does not contain the origin, since $G^\prime$ is not $(\ge 2, \ge 2)$.
    
    \item There is no pair of points $(\ge 2, \star), (\star,\ge 2)$.

    Segment between these points is not $B$-segment. And its affine span does not contain the origin since there are no points $(\ge 2, \ge 2)$.

\end{enumerate}

The first two observations provide that the only point outside coordinate cross is $(1,1)$, the last that on one of the axis the only allowed point is $1$. So we obtained the desired classification.
    
\end{proof}

\section{Marked $\mathbf{B}$-polytopes and proof of Lemma \ref{point}}\label{M-B-pol}

\textbf{Marked polytope} is a set $\alpha \in \Z^n$ of dimension $n$ with some \textbf{marked} facets.

\begin{definition}
    We call marked polytope $\alpha\in\Z^n$ ($dim(\alpha )= n$) a \textbf{marked }$\mathbf{B}$\textbf{-polytope} if every $n$-simplex with vertices in $\alpha$ is a \textbf{marked} $\mathbf{B}$\textbf{-simplex} in $\alpha$ i.e. there is a marked facet of $\alpha$ such that the $n$-simplex is a pyramid of lattice height 1 with base on the marked facet.
\end{definition}

Note that in contrast to Definition \ref{B_pol_def} here we use $n$-dimensional B-simplices (instead of $(n-1)$-dimensional B-simplices).

\begin{claim}
There are only two marked $B$-polytopes in dimension $1$ up to lattice isomorphism. They are segments of length $1$ with one or two marked vertices. 
\end{claim}

\begin{lemma}\label{section_lem}
    Consider a B-facet $\tau$ in $\Z^n_{\ge 0}$. Denote by $\D$ the $(n-1)$-dimensional simplex aff$(\tau) \cap \R^n_{\ge 0}$. Suppose that $\tau \cap E$ is a V-face but not a $B$-face in the coordinate subspace $E$. Consider an affine subspace $H \subset \text{aff}(\D)$ of dimension $n-1-dim(E)$, whose affine span with $E\cap \D$ is the whole $\D$. Suppose that $\tau_H = \tau \cap H$ is of the same dimension as $H$. We mark those facets of $\tau_H$ which are the intersections with coordinate hyperplanes $E_i$ containing $E$ such that $H$ contains lattice points with $i$-th coordinate equal to $1$.
    
    Then the set $\tau_H$ is a marked $B$-polytope. Moreover, if $\tau \setminus E = \tau_H$ then the conditions ``$\tau_H$ is a marked $B$-polytope'' and ``$\tau$ is a $B$-facet'' are equivalent.
\end{lemma}

Here are some illustrations of the notations from the latter lemma.

\begin{figure}[H]
    \centering
    \includegraphics[scale = 1]{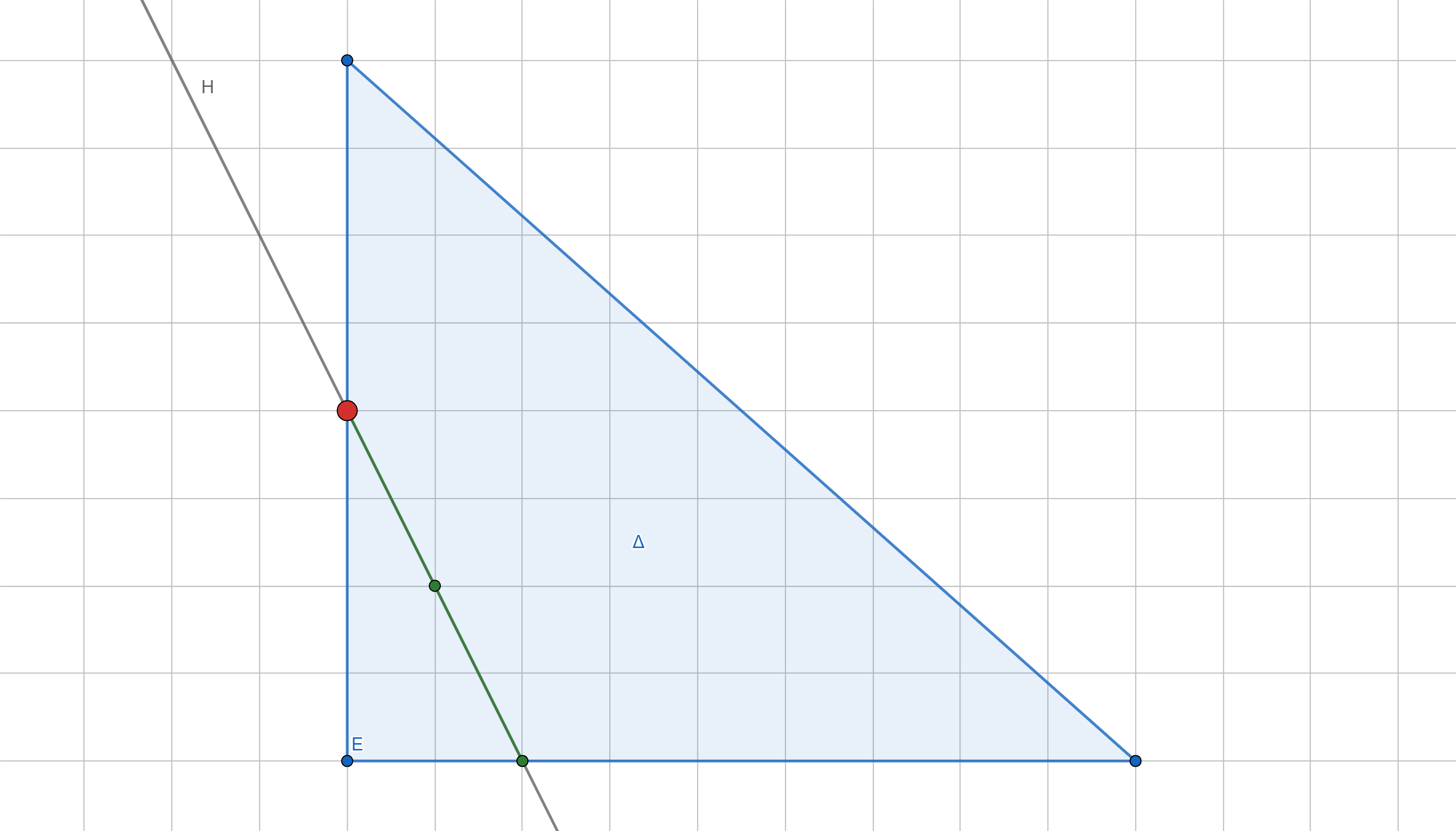}
    \caption{\label{marked2} $E$ is the ray containing vertex of right angle. There is only one marked vertex (the top one).}
\end{figure}

\begin{figure}[H]
    \centering
    \includegraphics[scale = .5]{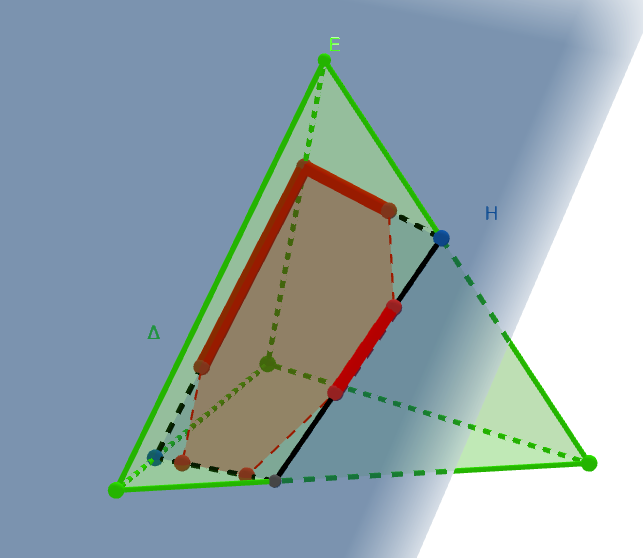}
    \caption{\label{marked} $E$ is the ray containing the top vertex of $\D$. There are at most $3$ marked sides.}
\end{figure}

\begin{proof}
    Consider a simplex $\D_E$ of dimension $dim(E) - 1$ in $E\cap \tau$ which is not a $B$-simplex. Consider a $dim(H)$-simplex $\D_H$ in $\tau_H$. Then their union $\D_E \cup \D_H$ is a $(n-1)$-simplex (because $dim(H) + dim(E) - 1 = n-1$ and aff$(H \cup (E\cap \D)) = \D$). The simplex $\D_E \cup \D_H$ is a $B$-simplex in $\Z^n_{\ge 0}$ if and only if $\D_H$ is a $B$-simplex in the marked polytope $\tau_H$. This equivalence implies Lemma \ref{section_lem}.
\end{proof}

Let us classify all marked $B$-polygons in dimension $4-2=2$. It is needed to prove Lemma \ref{point}.

\begin{lemma}\label{B_in_pol}
    Set $\alpha\subset \Z^2$ is a marked $B$-polygon if and only if (up to lattice transformation) one of the following cases is applied:
    \begin{enumerate}
        \item \textbf{$\mathbf{B_1}$-marked polygon}: $\alpha$ is pyramid with base on a marked side and height $1$.    \begin{figure}[H]
        \centering
        \includegraphics[scale = 2]{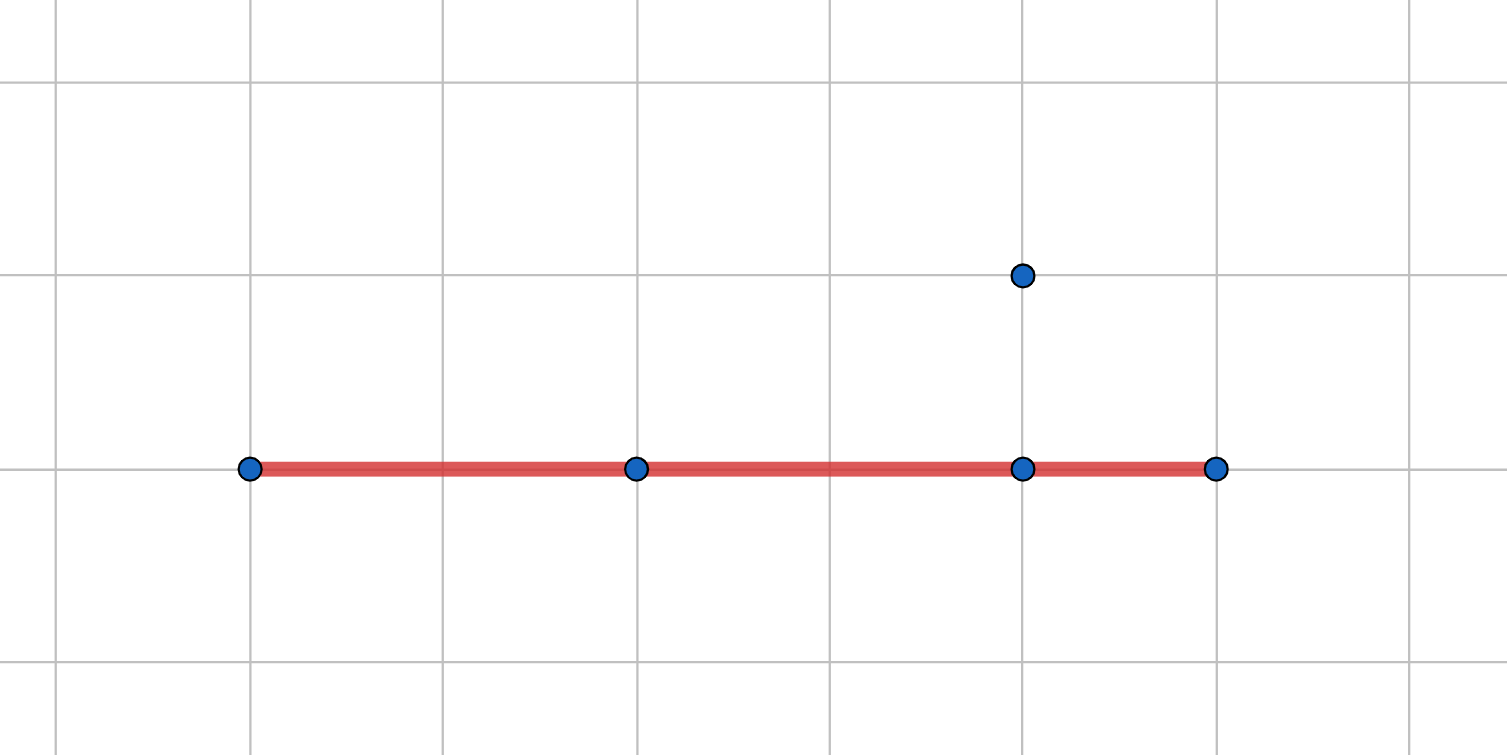}
        \caption{$B_1$}
        \label{B1-fig}
    \end{figure}
        \item \textbf{$\mathbf{B_2}$-marked polygon}: $\alpha$ is contained in the strip $x_1=\{0,1\}$ and both sides on the boundary of the strip are marked.
    \begin{figure}[H]
        \centering
        \includegraphics[scale = 2]{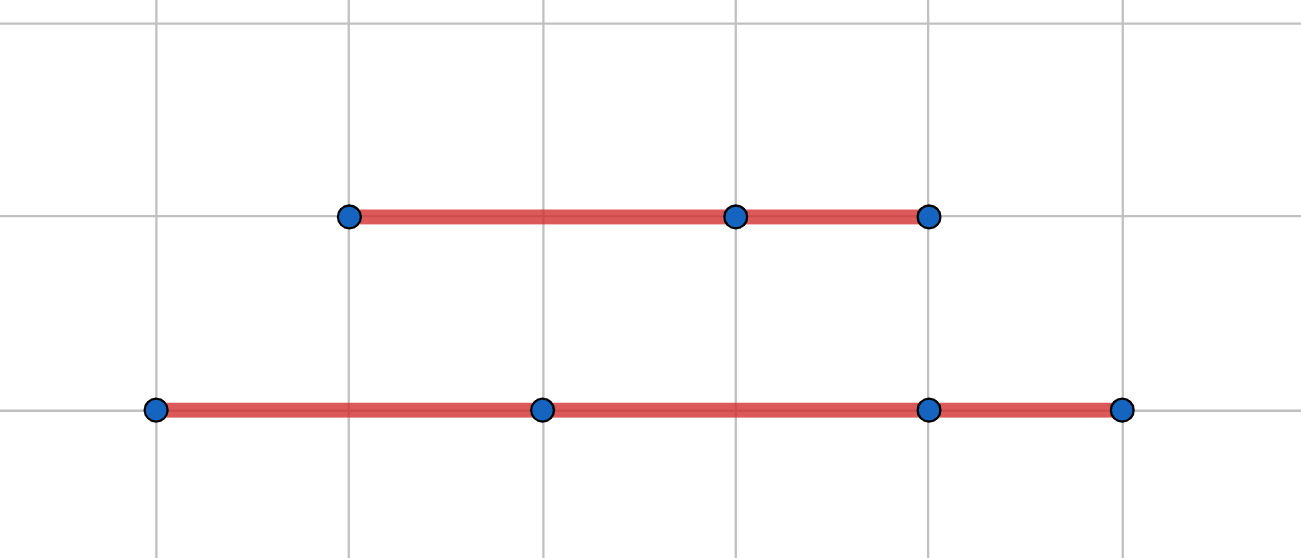}
        \caption{$B_2$}
        \label{B2-fig}
    \end{figure}
    
        \item \textbf{Flat border marked polygon}: $\alpha = \{(0,0),(\star>1,0), (0,1),(1,1)\}$ and all sides except $[(0,1),(1,1)]$ are marked.
        \begin{figure}[H]
        \centering
        \includegraphics[scale = 2]{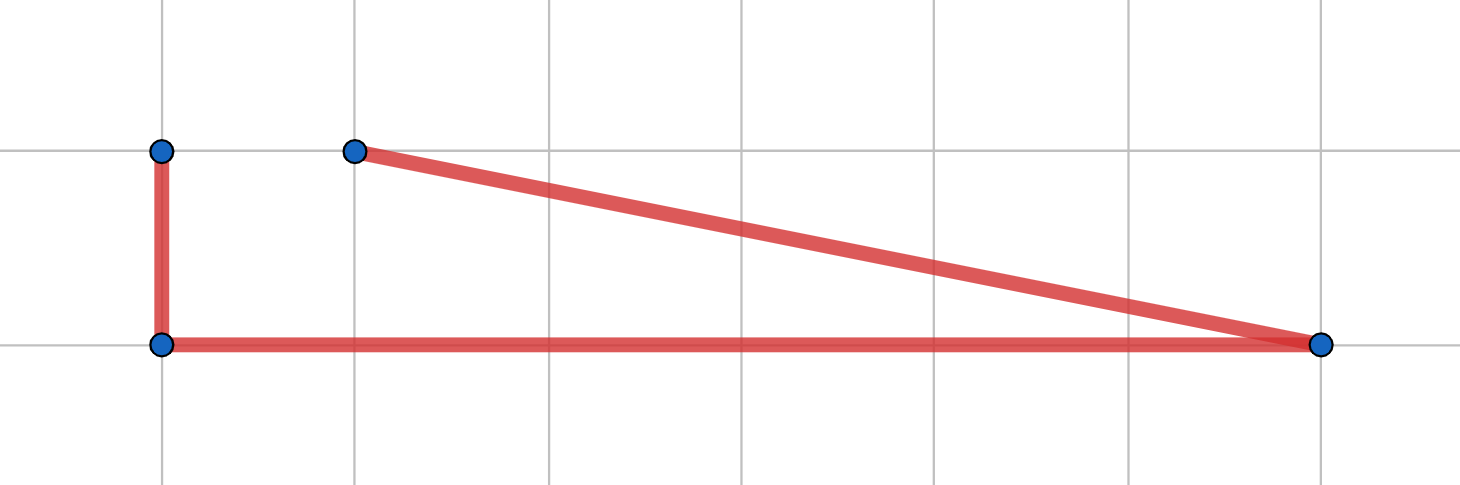}
        \caption{Flat border}
        \label{Flat-fig}
    \end{figure}
    \end{enumerate}
\end{lemma}

\begin{proof}
    Suppose $\alpha$ is not a $B_1$-marked polygon. Let us prove that there is a marked side of $P$ such that the distance from any point  of $\alpha$ to it is at most $1$. Assume the converse.
    Consider a $B$-triangle $AB\star \subset \alpha$ with marked base $AB$. Then there is a point $C\in \alpha$ such that $\rho(C,AB)>1$, where $\rho$ is the distance. Then $ABC$ must have another marked base, say, $BC$. Since $\rho(A,BC) = 1$ and $\rho(C,AB)>1$ the we have the following inequality for the lattice length $l(BC) > 1$. By our assumption, there is a point $D: \rho(D,BC)>1$ (see picture below).

    \begin{figure}[H]
        \centering
        \includegraphics[scale = 3]{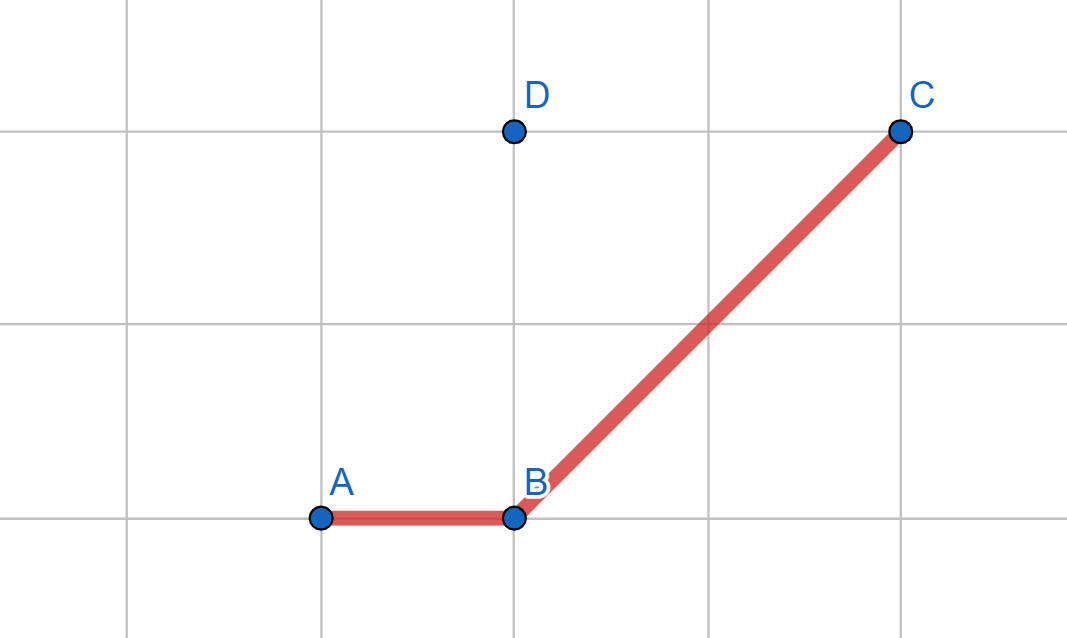}
        \caption{ABCD}
        \label{ABCD}
    \end{figure}

    Then the triangle's $BCD$ base is either $BD$ or $CD$. But both $\rho (B,CD), \rho(C,BD)$ are at least $l(BC)>1$. So $BCD$ is not $B$-triangle. Thus there is a marked side $L\subset P$ such that for any $Q\in \alpha$ the distance $\rho(Q,L) \le 1$. 
    
    If there are at least three points of $\alpha$ in $L$ or at least three points outside $L$ then points outside $L$ are in the same marked side of $P$ (the triangles on the picture below prove this) and $\alpha$ is a $B_2$-marked polygon case.

    \begin{figure}[H]
        \centering
        \includegraphics[scale = 3]{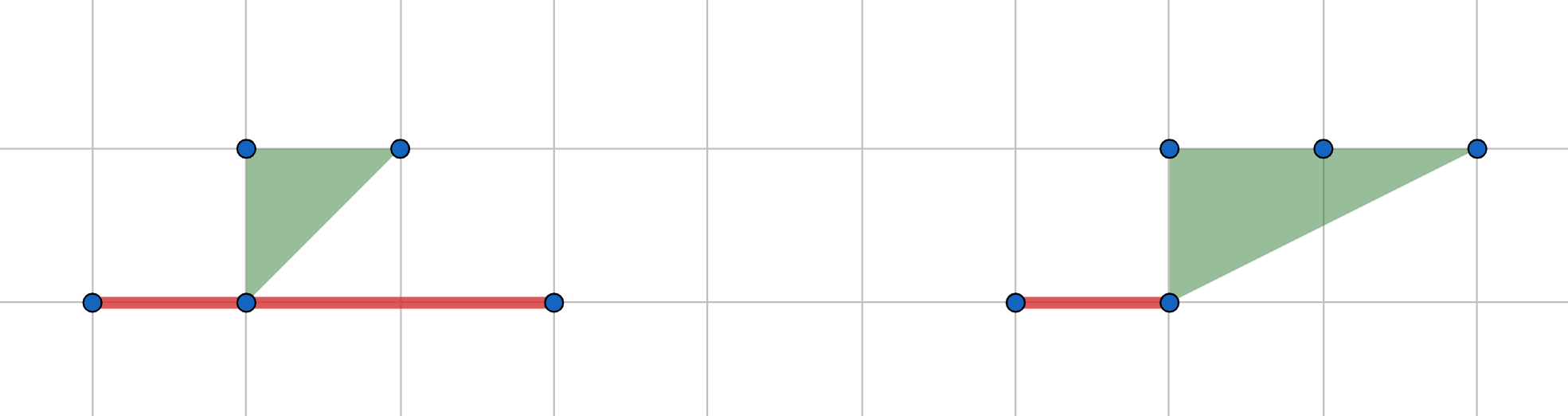}
        \caption{Triangles proving that top face is marked}
        \label{trianglesB2}
    \end{figure}

    The only case left is when $\alpha$ consists of $4$ points: two in $L$ and two outside are at height $1$ and do not form a marked side. After a lattice transformation we can assume that: $$\alpha = \{A,B,C,D\} = \{(0,0),(a,0), (0,1), (1,d)\}, \ a,d > 0 \quad L = \{(0,0),(a,0)\}.$$ Then bases of triangles $ACD$ and $BCD$ are marked sides $AC$ and $BD$ respectively. So $d = 1$ and there are two possibilities: $a = 1$ and $a > 1$. Case $a=1$ is the $B_2$-case (strip $x_2 = \{0,1\}$) and the case $a>1$ is the flat border case.
\end{proof}

Now we prove that Lemma \ref{B_in_pol} implies Lemma \ref{point}. Suppose that $B$-facet $\tau$ is a pyramid with apex $Q$ on a coordinate ray. Denote by $\tau_H$ the set defined in Lemma \ref{section_lem}. By Lemma \ref{section_lem} the polygon $\tau_H$ is a marked $B$-polygon, and these polygons are classified in Lemma \ref{B_in_pol}. In the cases of $B_1$- and $B_2$-marked polygons we obtain that $\tau$ is $B_1$- and $B_2$-faced respectively.

In the last case we have that $\tau_H$ is affine equivalent to a set

$$\alpha = \{A_H,B_H,C_H,D_H\} = \{(0,0),(a,0), (0,1), (1,1)\}, \ a > 0 \quad L = \{(0,0),(a,0)\}.$$

Then the coordinates of $\{A,B,C,D\}$ of these points in the initial space $\Z^4$ are of the form:

$$\{Q,A,B,C,D\} = \{(0,0,0,\star),(0,0,a,\star),(a,0,0,\star),(0,1,1,\star), (1,1,0,\star)\}.$$

Thus, we obtain a flat $B$-border pyramid from Example \ref{exotic_B_ex}.

\section{Proof of Lemmas \ref{triangle} and \ref{segment}}

\subsection{Proof of Lemma \ref{triangle}: internal triangle}

Suppose there is an internal V-triangle $A_1A_2A_3 \subset Ox_1x_2x_3$ in $\t$. 

\begin{enumerate}
    \item{Let us prove that $\tau \subset \{x_4 \le 1\}$}:\\
    Every $B$-tetrahedron with vertices $A_1A_2A_3$ has base $Ox_1x_2x_3$ since $A_1A_2A_3$ is internal and no $2$ of its points are contained in a coordinate 2-plane. So for all points in $\tau$ either $x_4 = 0$ or $x_4 = 1$.

    Assume that $\t$ is not a $B_1$-facet, i.e. it contains a segment $C_1C_2$ with $x_4 = 1$.

   \item   Let us prove that up to reordering of $A_i$ we have:
    $$A_1 = (> 0,> 0, 0, 0) \quad A_{2,3} = (\star,\star, 1, 0) \quad C_{1,2} = (\star,\star, 0, 1)$$
    
    \begin{figure}[H]
        \centering
        \includegraphics[scale = 0.3]{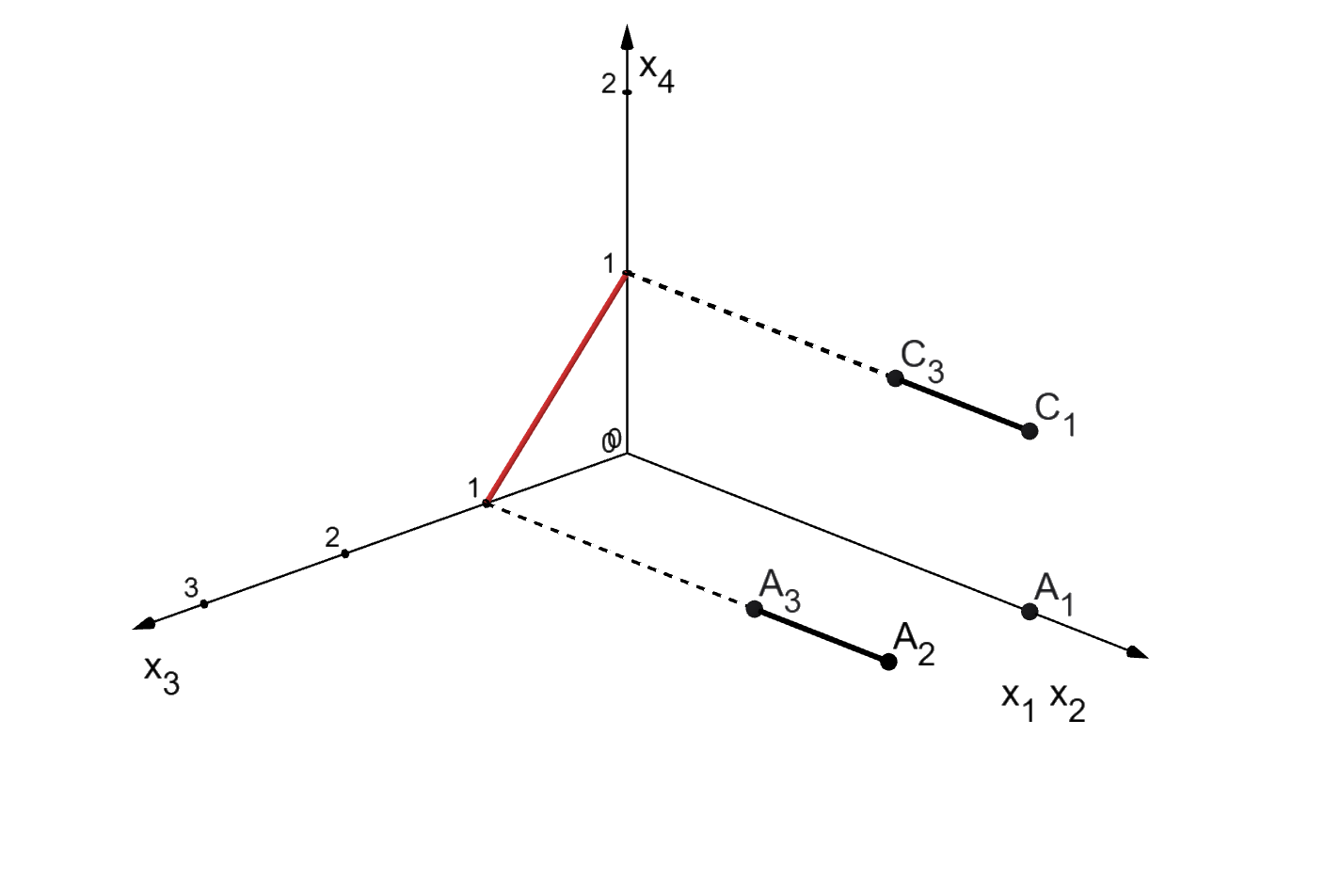}
        \caption{Internal triangle configuration}
        \label{triangle_config}
    \end{figure}
    
    Note that if $C_1C_2 A_iA_j$ is a tetrahedron then its apex is not $C_\alpha$ since $A_iA_j$ is not in a coordinate 2-plane. Note that least two of $C_1C_2 A_iA_j$ are tetrahedrons since the triangle $A_1A_2A_3$ can have at most one side parallel to $C_1C_2$. Without loss of generality we assume that $C_1C_2 A_1A_2$ is a tetrahedron with apex $A_2$ then $A_1 = (\star,\star, 0, 0), A_2 = (\star,\star, 1, 0), C_{1,2} = (\star,\star, 0, 1)$ and all the other coordinates of $A_1$ are positive since triangle $A_1A_2A_3$ is internal.
    
    Now we prove that $A_3 = (\star,\star, 1, 0)$. First note that $A_3 \ne (\star,\star, 0, 0)$, otherwise the side $A_1A_3$ is a $V$-segment. If $A_3 = (\star,\star, k, 0), k>1$ then $A_1 A_3C_1C_2$ is a tetrahedron, but not a $B$-tetrahedron ($x_{i} = 0,\ i = 1,2$ can't be a base because $A_1$ and some of $C_j$ are not contained there since $\t$ has positive normal covector). Thus, $A_3 = (\star,\star, 1, 0)$.
    
    \item Let us prove that $\t \subset \{x_3 \le 1\} \cap \{x_4 \le 1\}$:

    Suppose that there is a point $G$ such that $x_3(G) > 1$ then $A_1C_1C_2G$ is a tetrahedron but not $B$-tetrahedron (for the coordinate hyperplanes $x_{i} = 0,\ i = 1,2$ we use the same argument as in the previous item). If $x_4(G) > 1$ we consider $A_1A_2A_3G$ in the same way. 

    \item If there are no $G$ of the form $(\star,\star,1,1)$ then $\t$ is a (maybe degenerated) $B_2$-facet.
    
    \item If there is $G = (\star,\star,1,1) \in \t$ then $\t$ is the standard cross-polytope.

    Note that in one of the pairs $A_1GA_2C_1, A_1GA_3C_2$ and $A_1GA_2C_2, A_1GA_3C_1$ both polytopes are tetrahedrons (i.e. are not contained in a 2-dimensional affine subspace). These tetrahedrons are $B$-tetrahedrons, so $x_{1,2}$-coordinates of all points are also $0$ or $1$. It can be only if the segments $C_1C_2$ and $A_2A_3$ are shifted segment $[(1,0,0,0), (0,1,0,0)]$ since $\t$ has positive normal covector. Thus $A_1A_2A_3C_1C_2$ are $5$ of points of the standard cross-polytope and $G = (\star,\star ,1,1)$. So $G = (0,0,1,1)$ because it is the only point of this form contained in the affine hyperplane spanned by the others.

\end{enumerate}

\subsection{Proof of Lemma \ref{segment}: V-segment but not B-segment}

Suppose that $\tau$ contains a segment $A_1 A_2 \in Ox_1x_2$ which is not a $B$-segment. Then by Lemma \ref{proj_lem} the set $p_{Ox_1x_2} (\tau)$ is a $B$-polytope. We have two following cases by Lemma \ref{d-2}.

\begin{enumerate}

    \item \textbf{Border polytope}  

    Let us prove that in this case we obtain a flat border i.e. there is a point $C = (\star,\star,1,1)$ and all the other points are contained in $\{x_3= 0\} \cup \{x_4 = 0\}$. Moreover, we prove that (up to reordering $x_3$ and $x_4$), we have $x_4 \le 1$ for all the points.

    \begin{figure}[H]
        \centering
        \includegraphics[scale = 2]{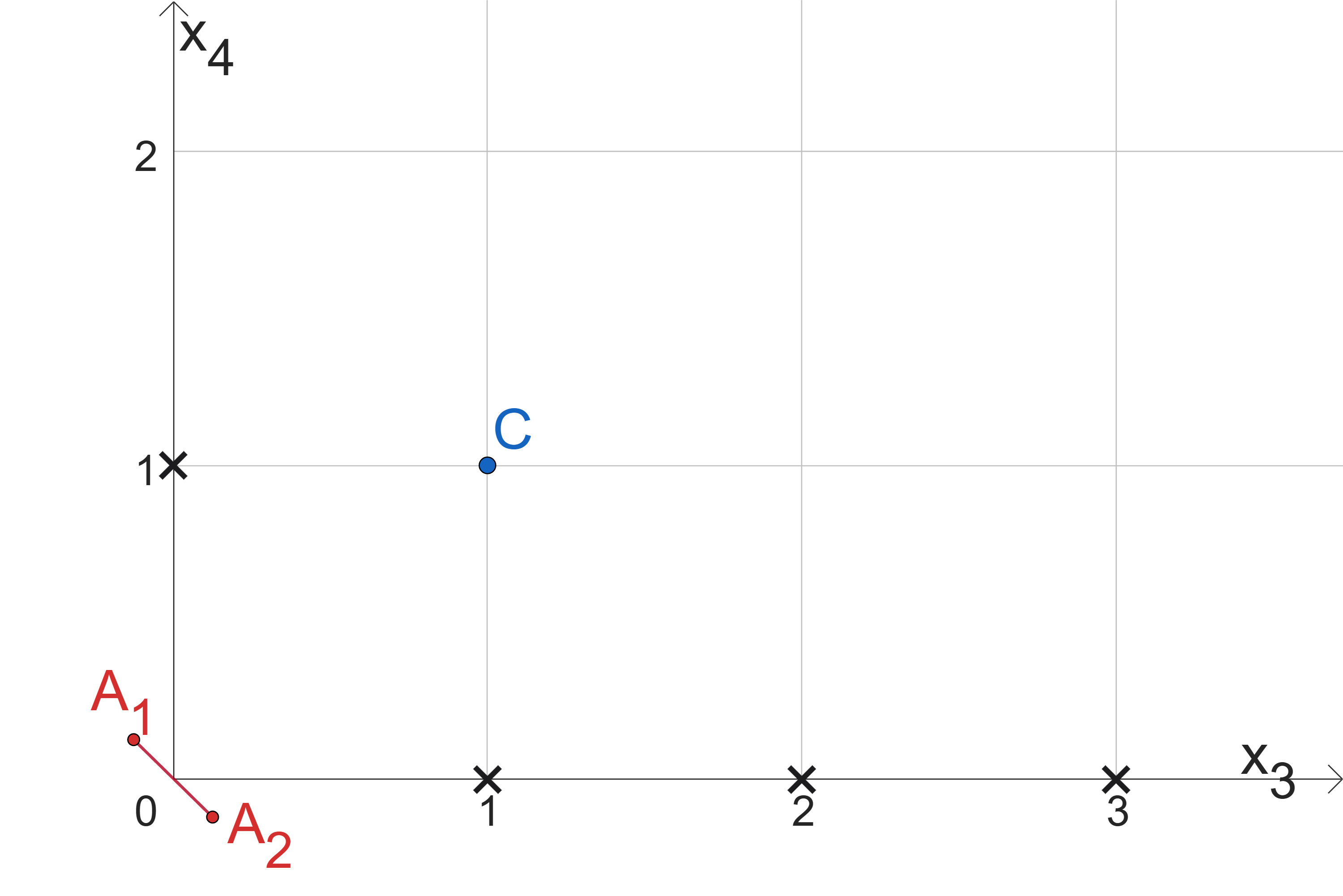}
        \caption{Projection of points}
        \label{Projection of points}
    \end{figure}
    
    The only thing we need to prove is that there is only one point $C$ of the form $(\star,\star,1,1)$. Suppose there two such points $C$ and $C^\prime$.
      
    Consider a point $D\in \tau$ of the form $(\star,\star,1,0)$. Then both $A_1CC^\prime D$ and $A_2CC^\prime D$ are tetrahedrons. They are $B$-tetrahedrons by the definition of $B$-facet. They cannot have base at $x_{3,4} = 0$ since both $C$ and $C^\prime$ have positive coordinates. They cannot have base at $x_{1,2} = 0$ for each $i=1,2$ one of $A_{1,2}$ and one of $C,C^\prime$ have positive coordinate (since $\t$ has positive normal covector). Thus, there is only one point $C$ of the form $(\star,\star,1,1)$.

    \item \textbf{$\mathbf{B_1}$-polytope}
    
    All the points are in $\{x_3 = 0\}$ except for some of the form $(\star,\star, 1, a)$ for some fixed $a$. Suppose there is only one point of the form $(\star,\star,a,1)$. Then $\tau$ is a $B_1$-facet. Assume the converse. Suppose there are at least two points $(C, C^\prime)$ of the form $(\star,\star,a,1)$

    \begin{enumerate}
        
        \item  Suppose that $a > 0$. Consider a point $D\in \tau$ of the form $(\star,\star, \ge 1, 0)$. Then both $A_1CC^\prime D$ and $A_2CC^\prime D$ are tetrahedrons. The only possibility for them to be $B$-tetrahedrons (up to reordering of first two coordinates) is 
        $$A_1 = (u,0,0,0), A_2 = (0,u,0,0), C_1 = (1,0,a,1),  C_2 = (0,1,a,1), D = (0,0,d,0),$$
        since $\tau$ has positive normal covector. Note that $u>1$ (again since $\tau$ has positive normal covector). Then $\tau$ contains the V-segment $A_2D \subset Ox_1x_4$ which is not $B$-segment. Then we obtain the first case and $\tau$ is flat border with triangle $A_2DC_1$.

        \item Suppose that $a=0$, and there is a point $D\in \tau$ of the form $(\star,\star, \ge 2, 0)$. Then the same argument from the previous item works.

        \item In the last case $a = 0$ and $\tau \subset \{x_3 \le 1\}$, the facet $\tau$ is a $B_2$-facet by the definition.
        
    \end{enumerate}    
    
\end{enumerate}

\vspace{1ex}
\noindent
\textit{Krichever Center for Advanced Studies}, Skolkovo Institute of Science and Technology, Moscow\\
\textit{Department of Mathematics}, National Research University ``Higher School of Economics'', Moscow \\
\textit{Email}: Fedor.Selyanin@skoltech.ru

\end{document}